\documentclass[graybox]{svmult}

\usepackage{amsfonts,amsmath,amscd,amssymb,mathtools,mathabx,rotating,latexsym,stmaryrd,mathabx}

\usepackage{mathptmx}       
\usepackage{helvet}         
\usepackage{courier}        
\usepackage{type1cm}        
\usepackage{makeidx}         
\usepackage{graphicx}        
\usepackage{multicol}        
\usepackage[bottom]{footmisc}

\makeindex             

\begin{document}


\title*{Quasi-shuffle algebras in non-commutative stochastic calculus}
\titlerunning{Karandikar's axiomatisation of matrix stochastic calculus}

\author{Kurusch Ebrahimi-Fard, Fr\'ed\'eric Patras}
\authorrunning{K.~Ebrahimi-Fard, F.~Patras} 

\institute{Kurusch Ebrahimi-Fard \at Department of Mathematical Sciences, 
		Norwegian University of Science and Technology -- NTNU, 
		7491 Trondheim, Norway, \email{kurusch.ebrahimi-fard@ntnu.no, \url{https://folk.ntnu.no/kurusche/}}
\and Fr\'ed\'eric Patras \at Universit\'e C\^ote d'Azur,
		Labo.~J.-A.~Dieudonn\'e,
         		UMR 7351, CNRS, Parc Valrose,
         		06108 Nice Cedex 02, France, \email{patras@unice.fr}, \url{www-math.unice.fr/$\sim$patras}}

		
\maketitle


\abstract{This chapter is divided into two parts. The first is largely expository and builds on
Karandikar's axiomatisation of It\^o calculus for matrix-valued semimartingales. Its aim is to unfold in detail the algebraic structures implied for iterated It\^o and Stratonovich integrals. These constructions generalise the classical rules of Chen calculus for deterministic scalar-valued iterated integrals. The second part develops the stochastic analog of what is commonly called chronological calculus in control theory. We obtain in particular a pre-Lie Magnus formula for the logarithm of the It\^o stochastic exponential of matrix-valued semimartingales.}

\abstract*{This chapter is divided into two parts. The first is largely expository and builds on
Karandikar's axiomatisation of It\^o calculus for matrix-valued semimartingales. Its aim is to unfold in detail the algebraic structures implied for iterated It\^o and Stratonovich integrals. These constructions generalise the classical rules of Chen calculus for deterministic scalar-valued iterated integrals. The second part develops the stochastic analog of what is commonly called chronological calculus in control theory. We obtain in particular a pre-Lie Magnus formula for the logarithm of the It\^o stochastic exponential of matrix-valued semimartingales.}


\section{Introduction}
\label{sec:intro}

Algebra, renormalisation theory as well as numerical analysis are among a range of disparate fields that have seen a surge in interest for the study of various algebraic and combinatorial structures originating from the integration by parts formula, foundational to integral calculus. In particular, the theories of Rota--Baxter algebras, shuffle and quasi-shuffle products, pre- and post-Lie algebras as well as combinatorial bialgebras on rooted trees and words have undergone expansive phases in the last two decades. The following rather incomplete list of references provides some examples of these developments  \cite{Brouder2000,Cartier2011,Calaque2011,ChaPa2013,FPT2016,FauMen2017,LM2011,MP2018a,MP2018b,MKW2007,Murua2007,go2008}. In the particular context of stochastic integration, interest in these structures concentrated largely in Lyon's seminal theory of rough paths \cite{Lyons1998,Lyons2007}, which is based on Chen's iterated path integrals and shuffle algebra on words \cite{Chen57,Chen71,Ree1958}. Gubinelli expanded Lyon's theory by generalising it to a certain combinatorial Hopf algebra of rooted trees. The resulting notion of branched rough paths \cite{Gubinelli2010,Gubinelli2011,HK2015} draws inspiration from Butcher's  theory of B-series in numerical integration of differential equations \cite{HLW2002,McLach2017} as well as Connes and Kreimer's Hopf algebraic approach to renormalisation in perturbative quantum field theory \cite{CK1998}. The latter, moreover emphasised the pre-Lie algebraic perspective on rooted trees \cite{ChapLiv2001,Manchon2011}. Ideas from rough paths gave rise to various new developments, culminating in Hairer's celebrated theory of regularity structures \cite{H2013,H2014} and its algebraic renormalisation theory \cite{BHZ2019} used in the construction of solutions of very irregular S(P)DEs.

The relevance of such algebraic structures for classical -- non-commutative -- stochastic integration (in the sense of It\^o--Stratonovich) \cite{BaldPlat2013,Protter2005} has attracted less attention. Foundational papers in this field are Gaines's 1994 work on the algebra of iterated stochastic integrals \cite{Gaines1994}, introducing what is now called quasi-shuffle product, as well as the equivalent sticky shuffle product formula for iterated quantum It\^o integrals introduced in 1995 in the context of quantum stochastic calculus \cite{Beasley1995}. We refer to Hudson's review papers on Hopf-algebraic aspects of iterated stochastic integrals \cite{Hudson2009,Hudson2012}. The authors together with Charles Curry, Alexander Lundervold, Simon Malham, Hans Munthe-Kaas and Anke Wiese further developed the use of (quasi-)shuffle algebra in stochastic integration theory and numerical methods for SDEs in several joint works \cite{CEFMW2014,CKP2018,CEFMW2019,EFLMMW2012,EMPW2015a,EMPW2015b}. 

The present article is divided into two main parts. The first part starts by recalling Karandikar's axiomatisation of It\^o calculus for matrix-valued continuous semimartingales \cite{Karandikar1982a}. We discuss in detail the algebraic structures implied for iterated It\^o and Stratonovich integrals. To the best of our knowledge, such an account does not seem to exist in the literature. In fact, Karandikar's ideas do not seem to be widely known, as the algebraic notions and techniques involved are not of common use in non-commutative stochastic calculus. Our presentation is written with a view toward operationality and therefore, as far as possible, in the language of theoretical probability theory. In modern algebraic terminology, Karandikar's axioms define the notion of non-commutative quasi-shuffle algebra. We proceed by building consistently on this structure aiming at unraveling properties of integration techniques, ranging from general matrix-valued semimartingales to more specific situations (continuous paths). We remark that introducing the Stratonovich integral in full generality requires extra axioms corresponding to the splitting of the quadratic covariation bracket into a continuous and a jump part. We hope that these ideas might be useful also in other settings. 
 
The second part develops the stochastic analog of what is commonly called chronological calculus in control theory \cite{AG1978,AS2004,ABB2019}. We feature in particular Agrachev and Gamkrelidze's \cite{AG1981,AGS1989} study and systematical use of chronological algebra in control problems. In mathematics, chronological algebras are known as pre-Lie or Vinberg algebras \cite{Burde2006,Cartier2011}. This part of our work is a continuation of joint work with Charles Curry \cite{CKP2018}, where a pre-Lie Magnus formula for the logarithm of the Stratonovich stochastic exponential for continuous matrix-valued semimartingales was introduced.

\smallskip 

A more detailed outline of the paper follows. The first section is divided into three subsections. We begin by focusing on general It\^o calculus for matrix-valued semimartingales, introduce Karandikar's axioms, and point out connections with other theories, especially Rota--Baxter algebras. The second subsection introduces formally the splitting of the covariation bracket, leading to a tentative axiomatisation of Stratonovich calculus. The last subsection studies stochastic calculus for continuous semimartingales, a situation where the axioms simplify dramatically, allowing to replace quasi-shuffles by shuffles, that is, permitting the use the standard rules of calculus.

In the second section we study pre-Lie algebraic aspects in stochastic integral calculus. We consider the pre-Lie Magnus formula in the general context of enveloping algebras of pre-Lie algebras -- the corresponding section can be understood also as an introduction to pre-Lie structures since we survey some of their most relevant properties for general integral calculus, following the chapter \cite{Nacer2020} written in the context of classical integration. The second and last subsection of this section changes focus by developing instead a pre-Lie point of view on It\^o integral calculus, without restriction to continuous matrix-valued semimartingales. We obtain in particular a pre-Lie Magnus formula for the logarithm of the It\^o stochastic exponential of matrix-valued semimartingales.

\smallskip

{\bf Conventions} i) With the aim of simplifying the presentation we shall always assume that the value of semimartingales is zero at $t=0$. ii) All structures are defined over a ground field $k$ of characteristic zero. 

\begin{acknowledgement}
The second author would like to express his gratitude for the warm hospitality he experienced during the 2019 Verona meeting, Random transformations and invariance in stochastic dynamics, with special thoughts for S.~Albeverio and S.~Ugolini. This work was supported by the French government, managed by the ANR under the UCA JEDI Investments for the Future project, reference number ANR-15-IDEX-01.
\end{acknowledgement}


\section{Karandikar's axioms and quasi-shuffle algebras} 
\label{sec:Itocalc}


\subsection{It\^o calculus for semimartingales} 
\label{ssec:Itocalc}

This section discusses the formal properties of the integral calculus for semimartingales. Protter's textbook \cite{Protter2005} will serve as the standard reference on stochastic integration. The central aim is to feature Karandikar's axioms for matrix-valued It\^o integrals and the corresponding notion of non-commutative quasi-shuffle algebra\footnote{Karandikar's axioms appeared in a 1982 paper \cite{Karandikar1982a}. They were (re-)discovered almost two decades later in a completely different context --the one of Stasheff polytopes-- as axioms for dendriform trialgebras \cite{LodayRonco2001}.}.     

Recall first It\^o's integration by parts formula for scalar semimartingales $X,Y$ \cite[chapter II.6]{Protter2005}
\begin{equation}
\label{eq:Ito1}
	X_tY_t= \int_0^tX_{s^-}dY_s + \int_0^tY_{s^-}dX_s + [X , Y]_t.
\end{equation}
Here as well as in the rest of the paper, our conventions are in place, that is, we assume that $X_0=Y_0=0$. Equation \eqref{eq:Ito1} defines the so-called quadratic covariation bracket, $[X , Y]_t$, the extra term that distinguishes It\^o's integration by parts formula from the classical one.

To deal with matrix-valued semimartingales, one has to take into account the non-commutativity of matrix multiplication. The product in the second term on the righthand side of \eqref{eq:Ito1} is then in the wrong order. Following Protter (and Karandikar), we introduce left and right stochastic integrals
\begin{equation}
\label{eq:leftright}
	(X \succ Y)_t:=\int_0^tX_{s^-}dY_s \qquad (X \prec Y)_t:=\int_0^tdX_sY_{s^-}.
\end{equation}
It\^o's integration by parts formula for matrix-valued semimartingales $X,Y$ writes then \cite[chapter V.8, Thm.~47]{Protter2005}
\begin{equation}
\label{eq:Ito2}
	X_tY_t = (X \succ Y)_t + (X \prec Y)_t + [X , Y]_t,
\end{equation}
Notice that in the case of scalar-valued semimartingales we have that $X\prec Y=Y\succ X$ such that \eqref{eq:Ito2} can be changed back to \eqref{eq:Ito1}.  

\begin{remark}\label{rmk:ProtterNotation}
Protter and Karandikar use a different notation: $(X \succ Y)_t=(X \cdot Y)_t=$ and $(X \prec Y)_t=(X : Y)_t$. Our notation is in line with the one used in algebra. It is also convenient to identify the time-ordering of operations, see \cite{Nacer2020}.
\end{remark}

Hereafter, we account for various properties of (left and right) stochastic integrals of scalar-valued semimartingales. Even though some of  the ternary formulas considered may seem redundant in the scalar case, we emphasise that they do not involve permutations of the variables. They thus hold immediately for ($n \times n$) matrix-valued semimartingales. 

Recall first \cite[Chapter II.6, Thm.~29]{Protter2005} that for $H,K$ adapted processes with caglad (left continuous with right limits) paths and $X,Y$ two semimartingales we have $[\int_0^tH_sdX_s,\int_0^tK_sdY_s]_t=\int_0^tH_sK_sd[X,Y]_s$. 

Assuming now that $L$ and $M$ are semimartingales, we get $[L\succ X,M\succ Y]=(LM)\succ [X,Y]$ which simplifies to the ternary relation
$$
	[L\succ X, Y]=L\succ [X,Y].
$$
Similarly, $[X\prec Y,Z]=[X,Y\succ Z]$ and $[X,Y]\prec Z=[X,Y\prec Z]$. 

Other ternary relations satisfied by semimartingales follow directly from standard properties and from the definitions: 
$$
	(X \succ (Y\succ Z))_t
	=\int_0^tX_{s^-}d(\int_0^tY_{s^-}dZ_s)
	=\int_0^tX_{s^-}Y_{s^-}dZ_s=((XY)\succ Z)_t.
$$
Similarly, for $G,H$ caglad and $Y$ a semimartingale, $\int_0^tG_{s}d(\int_0^tH_{s}dY_s)=\int_0^tG_{s}H_{s}dY_s$ \cite[Thm. II.19]{Protter2005}. Dually, $(X\prec Y)\prec Z=X\prec (YZ)$. 
Finally,
\begin{equation}
((X\succ Y)\prec Z)_t=\int\limits_0^t(d\int\limits_0^sX_{u^-}dY_u)Z_{s^-}=\int\limits_0^tX_{s^-}dY_sZ_{s^-}=(X\succ (Y\prec Z))_t.
\end{equation}
Notice that the associativity of the quadratic covariation bracket, $\left[X,\left[Y,Z\right]\right] = \left[\left[X,Y\right],Z\right],$ can be deduced from the associativity of the usual product of semimartingales together with the previous identities:
\allowdisplaybreaks
\begin{align*}
	\lefteqn{[[X,Y],Z]
	=[XY-X\prec Y-X\succ Y,Z]}\\
	&=(XY)Z- (XY)\prec Z-(XY)\succ Z-[X\prec Y+X\succ Y,Z]\\
	&=X(YZ)- X\prec (YZ)-X\succ (Y\prec Z)-X\succ (Y\succ Z)\\
	&\qquad -X\succ [Y,Z]-[X,Y\prec Z]-[X,Y\succ Z]\\
	&=[X,YZ]-[X,Y\prec Z]-[X,Y\succ Z]=[X,[Y,Z]].
\end{align*}

Note that, as mentioned before, $X,Y$ and $Z$ appear always in the same order in the above formulas. This implies that they hold in the matrix-valued case. In that case, $[X ,Y]$ is defined in terms of the component-wise quadratic covariation brackets, $[X ,Y]_{ik} = \sum_{j=1}^n [X_{ij},Y_{jk}]$, and similarly for the other products. Putting this together, yields the axiomatisation of It\^o calculus for semimartingales, due to Karandikar.

\begin{theorem}[Karandikar, \cite{Karandikar1982a}]\label{thm:quasishufIto} 
The left- and right stochastic It\^o integrals satisfy Karandikar's identities for matrix-valued semimartingales $X,Y,Z$
\begin{align}
\label{qsax1}	(X \prec Y) \prec Z  &= X \prec (YZ)\\		
\label{qsax2}	X \succ (Y \succ Z)  &= (XY) \succ Z\\	
\label{qsax3}	(X \succ Y) \prec Z  &= X \succ (Y \prec Z)\\
\label{assoc}	(X Y) Z  &= X (Y Z)\\
\label{qsax5}	\left[X \succ Y, Z\right]  &= X \succ \left[Y , Z\right]\\  
\label{qsax6}	\left[X \prec Y, Z\right]  &= \left[X,Y \succ Z\right]\\
\label{qsax7}	\left[X , Y\right] \prec Z  &= \left[X, Y \prec Z\right].
\end{align}
\end{theorem}

\begin{remark}
Karandikar considered in \cite{Karandikar1982a} the continuous case. Axiom \ref{qsax6} is stated in a slightly different way -- namely in the case where $Y$ is non-singular \cite[Eq.~(9) p.~1089]{Karandikar1982a}. However, these restrictions (to the continuous case and to non-singular $Y$) are not necessary, see Karandikar \cite{Karandikar1991}.
\end{remark}

\begin{definition}\label{def:quasi-shuffle} An associative space $A$ equipped with three products $\prec, \succ$ and $[\ ,\ ]$, called respectively the left half-shuffle, the right half-shuffle and the bracket, such that $XY=X\prec Y+X\succ Y+[X,Y]$ and satisfying Karandikar's identities in Theorem \ref{thm:quasishufIto} is called a quasi-shuffle algebra. 
\end{definition}

See Remark \ref{termQS} below for an explanation of the terminology. Karandikar's axiomatisation of It\^o calculus therefore says that the algebra of semimartingales is a non-commutative quasi-shuffle algebra. There are many examples of quasi-shuffle algebras besides the one coming from stochastic calculus and Karandikar's results apply immediately to them. Conversely, general results from abstract quasi-shuffle algebra apply to stochastic calculus. Examples of application of this strategy can be found in our joint works with Simon Malham and Anke Wiese \cite{EMPW2015a,EMPW2015b}.

\begin{remark}
We have seen that the associativity of the bracket operation $[\ ,\ ]$ 
\begin{equation}
	\label{qsax4}\left[X,\left[Y,Z\right]\right] = \left[\left[X,Y\right],Z\right]\\ 
\end{equation}
follows from the associativity of product in the algebra $A$. One can define equivalently a quasi-shuffle algebra to be a vector space $A$ equipped with three products $\prec, \succ$ and $[\ ,\ ]$ satisfying Equations (\ref{qsax1}-\ref{qsax7}) and \eqref{assoc} replaced by (\ref{qsax4}), where one sets $XY=X\prec Y+X\succ Y+[X,Y]$.
The associativity of the product $XY$ results then automatically from these axioms (the proof parallels the proof showing that the quadratic covariation bracket is associative for semimartingales):
\begin{align*}
	\lefteqn{(XY)Z
			=(XY)\prec Z+(XY)\succ Z+[X\prec Y+X\succ Y+[X,Y],Z]}\\
	&=X\prec (YZ)+X\succ (Y\prec Z)+X\succ (Y\succ Z) \\
	&\quad +[X,Y\prec Z]+[X,Y\succ Z]+X\succ [Y,Z]+[X,[Y,Z]]\\
	&=X\prec (YZ)+X\succ (Y\prec Z)+X\succ (Y\succ Z)+X\succ [Y,Z]+[X,YZ]=X(YZ).
\end{align*}
\end{remark}

\begin{remark}\label{termQS} The terminology ``quasi-shuffle'' algebra is used in algebra and combinatorics. It reflects the close similarity with the classical shuffle algebra \cite{Reutenauer1993}. The work \cite{FPT2016} explores the relation between the two families of algebras from a deformation theoretical viewpoint. Hoffman \cite{Hoffman2000}, independently of Karandikar's seminal work, largely initiated the development of the Hopf algebraic theory of commutative quasi-shuffle algebra. However, as we mentioned in the introduction, Gaines \cite{Gaines1994} as well as Hudson et al.~\cite{Beasley1995,Hudson2009,Hudson2012} introduced quasi-shuffle products to study properties of products of iterated It\^o integrals. We note that Cartier, back in 1972 \cite{Cartier2011}, used a quasi-shuffle product in the construction of free (Rota-)Baxter algebra. 

When dealing with scalar-valued semimartingales, we already noticed that $X\prec Y=Y\succ X$. Correspondingly, Karandikar's axioms simplify, leading to the notion of commutative quasi-shuffle algebra studied in detail by Hoffman, see \cite{FP2020} for a recent account. Some extra properties are then available such as Hoffman's isomorphism linking shuffle and quasi-shuffle products. Features of the commutative theory have been exploited recently in stochastic calculus, for example in \cite{CEFMW2014,EMPW2015b,Friedrich2016}.

From a purely algebraic viewpoint, the fundamental example of a quasi-shuffle algebra is the linear span of words $X=x_1\cdots x_n$ where the $x_i$ belong to a monoid $M$ with (not necessarily commutative) product denoted $\times$. The axioms for the products, $\prec$, $\succ$, and $[\ ,\ ]$, used to define inductively the associative product $X\ast Y:=X\prec Y+X\succ Y+[X ,Y]$ of words $X,Y$, are given by:
\begin{align*}
	x_1\cdots x_n \prec y_1\cdots y_m&:=x_1(x_2\cdots x_n\ast y_1\cdots y_m)\\
	x_1\cdots x_n \succ y_1\cdots y_m&:=y_1(x_1\cdots x_n\ast y_2\cdots y_m)\\
	[x_1\cdots x_n , y_1\cdots y_m]&:=(x_1\times y_1)(x_2\cdots x_n\ast y_2\cdots y_m).
\end{align*}
For example, taking $M$ to be the monoid of the integers, we have the quasi-shuffle product 
$$
	2\ 3\prec 1=2(3\ast 1)=2(3\prec 1+3\succ 1+[3,1])=2\ 3\ 1+2\ 1\ 3+2\ 4.
$$
\end{remark}

\begin{remark}
In the following, quasi-shuffle algebra shall mean non-commutative quasi-shuffle algebra. The latter are also called tridendriform algebras in the literature. However, the quasi-shuffle terminology, besides being close to other ones that have been used in stochastics (modified shuffle product, sticky shuffle product, ...) has the advantage of underlining the connection to the familiar shuffle calculus for Chen's iterated integrals and the related product of simplices, as well as many more topics (such as quasi-symmetric functions, multizeta values, etc.). We refer to \cite{EM2014,FP2020,NPT2013,Novelli2006} for accounts on the combinatorial theory of quasi-shuffle algebras as well as further bibliographical references and various examples. 
See \cite{CEFMW2014,CKP2018,EMPW2015a,EMPW2015b,Gaines1994,Hudson2009,Hudson2012} and references therein for more details and references on quasi-shuffle calculus in probability. 
\end{remark}

We introduce now Rota--Baxter algebras, which provide a more general approach to the algebraic axiomatisation of integral calculus \cite{Rota1998} and therefore an important class of examples for quasi-shuffle algebras. Indeed, Theorem \ref{thm:RBqs} below shows that any Rota--Baxter algebra is a quasi-shuffle algebra \cite{K2002}. We refer to the survey \cite{Nacer2020} for further details and references about Rota--Baxter algebras and their use in integral calculus, probability theory, renormalisation in perturbative quantum field theory and classical integrable systems. 

\begin{definition}\label{def:RBA} 
A {\rm{Rota--Baxter algebra of weight $\theta \in k$}} consists of an associative $k$-algebra $A$ equipped with a linear operator $R \colon A \to A$ satisfying the {\rm{Rota--Baxter relation of weight $\theta$}}:
\begin{equation}
\label{RBR}
	R(x)R(y)=R\big(R(x)y+xR(y)+\theta xy \big) \qquad \forall x,y \in A.
\end{equation}
\end{definition}
Note that if $R$ is a Rota--Baxter map of weight $\theta$, then the map $R':=\beta R$ for $\beta \in k$ different from zero is of weight $\beta \theta$. This permits to rescale the original weight  $\theta \neq 0$ to the standard weight $\theta' = +1$ (or $\theta' = -1$). The argument of the map $R$ on the righthand side of \eqref{RBR} consists of a sum of three terms; one can show that it defines a new associative product on $A$. 

\begin{definition}[Rota--Baxter product]
The Rota--Baxter associative product is defined by
\begin{equation}
\label{doubleasso}
	x \ast_{\!\scriptscriptstyle{\theta}} y := R(x)y + xR(y) + \theta xy.
\end{equation}
\end{definition}

The Rota--Baxter relation originated in the work of the mathematician Glen Baxter \cite{Baxter1960}. Rota \cite{Rota1969,Rota1995} followed by Cartier \cite{Cartier1972} made important contributions to the algebraic foundations of Baxter's work, among others, by providing different constructions of free commutative objects. The idea of quasi-shuffle product is actually often traced back to Cartier's 1972 article.

\begin{theorem}[\cite{K2002}]\label{thm:RBqs} 
Assume now that $\theta=1$. Writing $a \cdot b:=ab$ for the usual associative product on $A$ and setting $a\prec b:=aR(b)$ and $a\succ b:=R(a)b$, so that $\ast:=\prec +\succ +\ \cdot$, then the quasi-shuffle algebra identities hold:
\begin{equation}
\label{quasishuffleNC}
\begin{tabular}{ l l }
	$(a \prec b) \prec c = a \prec (b\ast c)$, 		&\quad $(a \succ b) \cdot c = a \succ (b \cdot c)$ \\  
	$a \succ (b \succ c) = (a\ast b) \succ c$, 		&\quad  $(a \prec b) \cdot c = a \cdot (b \succ c)$ \\
	$(a \succ b) \prec c = a \succ (b \prec c)$, 	&\quad  $(a \cdot b) \prec c = a \cdot (b \prec c)$. 
\end{tabular}
\end{equation}
\end{theorem}

\begin{remark}
Without the normalisation to the standard weight, one obtains an example of a $\theta$-quasi-shuffle algebra, studied in greater detail in \cite{BR2010}. See also \cite{HoffmanIhara2016}. 
\end{remark}

\begin{example}[Fluctuation theory] Baxter's work was motivated by problems in the theory of  fluctuations \cite{Rota1972}. The latter deals with extrema of sequences of real valued random variables. Their distribution can be studied using operators on random variables such as $X\to X^+:=\max(0,X)$. This motivates to define the operator 
$$
 	R(F)(t):={\mathbf E}[\exp({itX^+})]
$$
on characteristic functions $F(t):={\mathbf E}[\exp({itX})]$ of real valued random variables $X$, which is a Rota--Baxter map of weight $\theta=1$. 
\end{example} 

\begin{example}[Finite summation operators]
On functions $f$ defined on $\mathbb{N}$ and with values in an associative algebra $A$, the summation operator $R(f)(n):=\sum_{k=0}^{n-1}f(k)$ is a Rota--Baxter map of weight one. It is the right inverse of the finite difference operator $\Delta(f)(n):=f(n+1) - f(n)$.       
\end{example}

\begin{remark}[Shuffle algebras in classical calculus]
Before concluding this section, we apply the previous ideas to the case of deterministic matrix-valued semimartingales. Even in that seemingly simple case the relations put forward by Karandikar prove to be interesting and useful.

We consider for example the algebra $A$ of matrices whose entries are continuous functions of finite variation. Then, since the quadratic covariation bracket vanishes, Karandikar's identities reduce to an algebra equipped with an associative product $XY=X \prec Y+X \succ Y$ and 
\begin{align}
\label{shuffle}
\begin{split}
	(X \prec Y) \prec Z &= X \prec (YZ) 		\\
	X \succ (Y \succ Z) &= (XY) \succ Z		\\
	(X \succ Y) \prec Z &= X \succ (Y \prec Z).
\end{split}
\end{align}
Our previous arguments show that the associativity of the product $XY$ can be recovered formally from these identities. These relations have been used first by Eilenberg and MacLane to give an abstract proof of the associativity of the shuffle product of simplices in topology. 

\begin{definition}\label{shu}
The three identities \eqref{shuffle} define the structure of non-commutative {\rm shuffle algebra} (aka dendriform algebra).
\end{definition}

From Theorem \ref{thm:RBqs} is clear that on a Rota--Baxter algebra of weight $\theta=0$ one can define left and right half-shuffle products satisfying the three identities \eqref{shuffle}. We refer to \cite{Nacer2020} for a survey and more details and applications in classical integral calculus as well as general references on the subject. We will come back to these relations later as they encode the algebra structure underlying Stratonovich calculus for semimartingales.
\end{remark}


\subsection{Singular quasi-shuffle algebras and Stratonovich calculus}
\label{ssec:singular}

In this section we extend Karandikar's axiomatisation beyond the setting of continuous semimartingales. Namely, we include in the algebraic description of It\^o calculus the decomposition into continuous and jump parts of the quadratic covariation bracket \cite[Chapter II.6]{Protter2005}. 

We write $\Delta(X)$ for the process of jumps of a semimartingale $X$, i.e., $\Delta(X)_s=(X-X_-)_s$, and introduce the corresponding decomposition of the bracket into continuous and jump parts,
$[X,Y]=[X,Y]^c+[X,Y]^j$. The definition extends from the scalar-valued to the continuous matrix-valued case components-wise. For scalar-valued processes, $[X,X]_t^j=\sum_{0\leq s\leq t}(\Delta(X)_s)^2$, a term that appears frequently in stochastic calculus, for example,  in the study of the stochastic or Dol\'eans-Dade exponential \cite[chapter II.8, Thm.~37]{Protter2005}. A semimartingale $X$ is called quadratic pure jump if $[X,X]=[X,X]^j$.

Recall first a fundamental property of $\Delta$ acting on scalar-valued processes. 
Since the bracket of two semimartingales has paths of finite variation on compact sets \cite[chapter II, Cor.1]{Protter2005}, it is a quadratic pure jump semimartingale, that is, $[[X,Y],[X,Y]]=[[X,Y],[X,Y]]^j$ by \cite[chapter II.6, Thm.~26]{Protter2005}. This implies by \cite[chapter II.6.~Thm. 28]{Protter2005} that for arbitrary semimartingales $X,Y,Z$ we have $[[X,Y],Z]=\sum_{0\leq s\leq t}\Delta([X,Y])_s\Delta(Z)_s.$ In particular, 
\begin{align*}
	[[X,Y]^c,Z]&=[[X,Y]^c,Z]^c=[[X,Y]^c,Z]^j=0,\\
	[[X,Y],Z]^c&=[[X,Y]^j,Z]^c=0,\\
	[[X,Y],Z]    &=[[X,Y]^j,Z]=[[X,Y]^j,Z]^j.
\end{align*}
As a corollary, we notice for further use that for continuous semimartingales $[[X,Y],Z]=0$. These identities hold for matrix-valued semimartingales (since the splitting of processes into a continuous and a pure jump part is linear -- it commutes with taking linear combinations of brackets).

A full axiomatisation of It\^o calculus taking into account such phenomena would require the introduction of the operator $\Delta$, those identities, and most likely other aspects of standard stochastic calculus. We propose a lighter version that provides an axiomatic framework allowing to relate formally It\^o and Stratonovich calculi.

\begin{definition}
A singular quasi-shuffle algebra is a quasi-shuffle algebra, $(A,\succ,\prec,[-,-])$, such that the associative bracket splits into $[-,- ]=[-,- ]^c+[-,- ]^j$ and furthermore the following relations hold:
\begin{align}
\label{qsax8} [[X,Y]^c,Z]^c  &=  [X,[Y,Z]^c]^c=0\\
\label{qsax9} [[X,Y]^c,Z]^j  &=  [X,[Y,Z]^c]^j=0\\
\label{qsax10} [[X,Y]^j,Z]^c &=  [X,[Y,Z]^j]^c=0 
\end{align}
\end{definition}
Notice that we also have then
$$[X\prec Y,Z]^c=[X,Y\succ Z]^c,\ 
 [X\prec Y,Z]^j=[X,Y\succ Z]^j.$$

Recall that for matrix-valued semimartingales, the (left/right) Fisk--Stratonovich integrals are defined in terms of the It\^o integral by 
\begin{align}
\label{eq:FStratoLeft}
	(X \succcurlyeq Y)_t:=\int_0^t X_s \circ\!dY_s &:= \int_0^t X_s dY_s + \frac{1}{2} [X, Y]^c_t\\
\label{eq:FStratoRight}
	(Y \preccurlyeq X)_t:=\int_0^t \circ dY_s X_s &:= \int_0^t dY_s X_s + \frac{1}{2} [Y, X]^c_t	.
\end{align}
Formally, in any singular quasi-shuffle algebra one can define the two products
$$
	X \succcurlyeq Y:=X\succ Y+\frac{1}{2}[X,Y]^c,
	\qquad
	X \preccurlyeq Y:=X\prec Y+\frac{1}{2}[X,Y]^c.
$$
Then the integration by parts rule reads:
\begin{equation}
	XY=X \succcurlyeq Y+X \preccurlyeq Y+[X,Y]^j.
\end{equation}
Unfortunately, it seems to be difficult to find a simpler axiomatic framework than the one of singular quasi-shuffle algebras to account for Stratonovich calculus in the presence of jumps. Indeed, it is likely that a meaningful system of ternary relations involving only $\succcurlyeq$, $ \preccurlyeq$ and $[\ ,\ ]^j$ is unavailable. Fortunately, these issues simplify considerably for continuous semimartingales.


\subsection{Shuffle algebra and continuous semimartingales}
\label{ssec:contsemim}

As we mentioned previously, stochastic integration simplifies dramatically when considering continuous semimartingales. The reason for this should be clear from our previous developments, that is, the jump part, $[- , -]^j$, of the bracket, $[- , -]$, vanishes, so that the latter reduces to its continuous part and becomes nilpotent of order 3: $[X,[Y,Z]]=[[X,Y],Z]=0$. In that situation, the axioms of It\^o calculus rewrite: 

\begin{lemma}
Continuous matrix-valued semimartingales equipped with the left and right half-shuffles and the covariation bracket obey the axioms (\ref{qsax1})-(\ref{qsax7}) together with
\begin{equation}
	\label{qsax12} [X,[Y,Z]]=[[X,Y],Z]=0
\end{equation}
\end{lemma}

Notice that the associativity of the product is formally a consequence of these axioms. This observation is of little interest when dealing with stochastic integrals for which the associativity of the product is somehow obvious. However, it is relevant with respect to the axiomatic point of view.

\begin{definition}
A regular quasi-shuffle algebra is a quasi-shuffle algebra such that the bracket satisfies the extra axiom (\ref{qsax12}).
\end{definition}

The continuity hypothesis has more interesting consequences when dealing with Fisk--Stratonovich integrals. We follow closely the exposition in \cite{CKP2018}. The Stratonovich formula is indeed then the usual integration by parts formula
\begin{equation}
\label{eq:Strato1}
	X_tY_t = (X \succcurlyeq Y)_t + (X \preccurlyeq Y)_t,
\end{equation}

The classical statement that the Stratonovich integral for continuous semimartingales obeys the usual laws of calculus translates formally into the

\begin{theorem}\label{thm:shuffleStrato}
For continuous semimartingales $X,Y,Z$, the left and right Fisk--Strato- novich integrals satisfy the half-shuffle identities 
\begin{align}
\label{QS1a} (X \preccurlyeq Y) \preccurlyeq Z &= X \preccurlyeq (YZ) \\
\label{QS2a} (X \succcurlyeq Y) \preccurlyeq Z &= X \succcurlyeq (Y \preccurlyeq Z)\\
\label{QS3a} X \succcurlyeq (Y \succcurlyeq Z) &= (XY) \succcurlyeq Z.
\end{align}
In particular, the algebra of continuous matrix-valued semimartingales is a non-commutative shuffle algebra.
\end{theorem}
\begin{proof}
\begin{align*}
	\lefteqn{(X \preccurlyeq Y) \preccurlyeq Z= \big(X \prec Y + \frac{1}{2} [X, Y] \big) \preccurlyeq Z}\\ 
		&= (X \prec Y) \prec Z + \frac{1}{2} [X,Y] \prec Z 
			+  \frac{1}{2} [X \prec Y, Z] + \frac{1}{4} [[X, Y], Z] \\
		&= X \prec (YZ) + \frac{1}{2} [X,Y \prec Z] + \frac{1}{2} [X ,Y \succ  Z] \\
		&= X \preccurlyeq (YZ) .
\end{align*}
Identity (\ref{QS3a}) is proved similarly.
\begin{align*}
	\lefteqn{(X \succcurlyeq Y) \preccurlyeq Z = \big(X \succ Y + \frac{1}{2} [X, Y]\preccurlyeq Z }\\
		&= (X \succ Y) \prec Z + \frac{1}{2} [X, Y] \prec Z +  \frac{1}{2} [X \succ Y, Z]+ \frac{1}{4} [[X, Y], Z]\\
		&= X \succ (Y \prec Z) +  \frac{1}{2} [X,Y \prec Z] + \frac{1}{2} X \succ [Y, Z] \\
		&= X \succcurlyeq (Y \preccurlyeq Z).
\end{align*}
\end{proof}
In general, the same argument show 
\begin{theorem}
The map $(A,\prec,\succ,[\ ,\ ])\longmapsto (A,\preccurlyeq,\succcurlyeq)$ is a functor from the category of regular quasi-shuffle algebras to the category of shuffle algebras.
\end{theorem}


\section{Chronological calculus for stochastic integration}
\label{sec:chronostoch}

In this section, the second part of this work, we start by briefly reviewing the classical chronological calculus following Agrachev, Gamkrelidze and collaborators \cite{AG1978,AG1981,AGS1989,AS2004}. The aim is to show how chronological calculus can be applied in the context of stochastic calculus. The key idea is to use the notion of pre-Lie (or chronological) algebra instead of that of usual Lie algebra, to analyse group- and Lie-theoretical phenomena associated to evolution equations. We refer to \cite{Nacer2020} where this point of view is developed in more detail.


\subsection{Chronological calculus and pre-Lie algebra}
\label{ssec:chronostoch}

Time- or path-ordered products are ubiquitous, especially in theoretical physics and control theory, and form the basis for Agrachev and Gamkrelidze's chronological calculus \cite{AG1978}. These  authors understood that the combination of Lie algebra and integration by parts permits to define the useful notion of chronological algebra \cite{AG1981}, better known as pre-Lie or Vinberg algebra in algebra and geometry \cite{Burde2006,Cartier2011,Manchon2011}. Concepts from chronological calculus apply in the context of stochastic integration as far as iterated Stratonovich integrals for continuous semimartingales are concerned, because they obey the usual rules of calculus. 

However, for It\^o and Stratonovich integrals in the non-continuous case, the usual ideas of chronological calculus do not apply immediately, due to the terms arising from the jump component of the covariation bracket. It turns out that in this case, one must appeal to results originating in the study of non-commutative Rota--Baxter algebras. We refer to \cite{Nacer2020} for more details as well as to joint works \cite{CKP2018,EMPW2015b} for results in that direction related to stochastic exponentials in the context of It\^o calculus.

\smallskip

In a nutshell, chronological calculus is based on the idea of time-ordering of operators. Consider for example two time-dependent operators, $M(t)$ and $N(t)$ (with $M(0)=N(0)=0$), in a non-unital algebra $A$ of operators -- having suitable regularity properties allowing to compute derivatives, integrals, and so on. The classical integration by parts rule is satisfied
\begin{align*}
	M(t)N(t)
	&= \int_0^t ds \int_0^s du \dot{M}(s)\dot{N}(u) + \int_0^tds \int_0^s du \dot{M}(u)\dot{N}(s)\\
	&=:(M \succ N)(t)+(M \prec N)(t).
\end{align*}
We recognise in $\prec$ and $\succ$ the usual operations of left/right integration -- restricted to the context of deterministic processes. In particular, they satisfy the shuffle algebra axioms \eqref{shuffle}. Agrachev and Gamkrelidze observed that the binary operation 
$$
	(M \triangleright N)(t) 
	:= (M \succ N)(t)- (N \prec n)(t) 
	= \int_0^t ds \int_0^s du [\dot{M}(s),\dot{N}(u)]
$$
has particular properties defining a chronological algebra structure on $A$. The latter is known as pre-Lie or Vinberg algebra in the mathematical literature. 

Consider a vector space  $A$ with a binary product $\triangleright \colon A \otimes A \to A$ and the associated bracket product $[a,b]_\triangleright := a \triangleright b-b \triangleright a$. Write $L_x$ for the linear endomorphism of $A$ defined by left multiplication, $L_x(y):=x\triangleright y,$ and define the usual commutator bracket of linear endomorphisms of $A$, $[L_x,L_y]_\circ :=L_x\circ L_y-L_y\circ L_x$.

\begin{definition}[\cite{AG1981}] The pair $(A,\triangleright)$ is a pre-Lie algebra if and only if for any $x,y\in A$, the identity $[L_x,L_y]_\circ= L_{[x,y]_\triangleright}$ holds, which is equivalent to the (left) pre-Lie relation   
$$
	x \triangleright (y\triangleright z) - (x\triangleright y)\triangleright z
	=y\triangleright (x\triangleright z) - (y\triangleright x)\triangleright z.
$$
\end{definition}
The notion of pre-Lie algebra is finer than that of Lie algebra (it contains more information). Indeed, pre-Lie algebras are Lie admissible, that is, if $A$ is a pre-Lie algebra, then $(A,[ -, - ]_\triangleright)$ is a Lie algebra. 

The link with classical chronological calculus is as follows. Consider an algebra $\mathcal A$ of matrix-valued continuous semimartingales equipped with the left/right Fisk--Stratonovich integrals, $\succcurlyeq$ and $\preccurlyeq$, defined in \eqref{eq:FStratoLeft} respectively \eqref{eq:FStratoRight}. We write $\llbracket - , - \rrbracket$ for its commutator bracket: $$\llbracket X,Y\rrbracket _t
	:=X_tY_t-Y_tX_t.$$
According to our previous developments, computing in this algebra amounts to computing with time-dependent operators. The -- Fisk--Stratonovich -- integration by parts formula implies that 
\begin{equation}
\label{assocprod}
	\llbracket X,Y\rrbracket _t
	=\int_0^tX_s\circ dY_s+\int_0^t\circ dX_t Y_t
				-\int_0^tY_s\circ dX_t - \int_0^t\circ dY_sX_s,
\end{equation}
which can be written as the difference of: 
$$
	(X\triangleright Y)_t:=(X \succcurlyeq Y - Y\preccurlyeq X)_t
	=\int_0^tX_s\circ dY_s - \int_0^t\circ dY_sX_s
$$
and $(Y\triangleright X)_t$ so that $\llbracket X,Y\rrbracket =[X,Y]_\triangleright.$ That the algebra $\mathcal A$ is indeed a pre-Lie algebra, that is,
\begin{align}
\label{preLieid}
	([X,Y]_\triangleright\ \triangleright Z)_t
	&=(X\triangleright (Y\triangleright Z))_t - (Y\triangleright (X\triangleright Z))_t.
\end{align}
follows from the Jacobi identity of the commutator bracket on the non-commutative algebra $\mathcal A$. 

\begin{remark}
The same argument shows that, more generally, there is a forgetful functor from shuffle to pre-Lie algebras, that is, any shuffle algebra $(A,\prec,\succ)$ has the structure of a pre-Lie algebra with pre-Lie product: $x \triangleright y := x \succ y- y \prec x.$
\end{remark}

Let us apply these ideas in the context of stochastic exponentials. Recall a fundamental object in the classical analysis of differential equations, known as the Magnus formula \cite{Magnus1954} and its pre-Lie interpretation \cite{AG1981,EM2009}. It follows from studying the formal properties of the flow associated to a matrix differential equation using a Lie theoretic approach, for theoretical and numerical reasons. Consider for instance the evolution operator solution of the linear differential equation $\dot{X}(t)=X(t)H(t)$ with initial value $X(0)=\mathbf 1$, the identity matrix. Its logarithm is computed by a Lie series. Truncating the expansion of this logarithm, $\Omega(t):=\log(X(t))$, and taking its exponential is a classical and efficient way to approximate $X(t)$ numerically, while preserving group-theoretic properties \cite{Blanes2008,Iserles2000}.

The logarithm can be computed using the Baker--Campbell--Hausdorff formula (see, e.g., \cite{MielPleb1970}) or Magnus' non-linear differential equation \cite{Magnus1954}  
$$
	\dot{\Omega}(t)
	=\frac{ad_{\Omega}}{e^{ad_{\Omega}}-1}H(t)
	=H(t)+\sum\limits_{n>0}\frac{B_n}{n!}ad_{\Omega(t)}^n(H(t)),
$$
where $ad$ stands for the usual Lie adjoint representation, $ad_N(M):=NM-MN$, $ad^0_N(M)=M$, and the $B_n$ are the Bernoulli numbers. 

Let us explain how the formula adapts to Stratonovich integrals using recently developed algebraic tools that are most likely not familiar in the context of stochastic integration.  The following developments are based on \cite{Nacer2020}. We omit here the group-theoretical perspective that relies on two underlying Hopf algebra structures, existing on the enveloping algebra of any pre-Lie algebra \cite{ChaPa2013}. 

Let $\mathcal A$ be our usual algebra of continuous matrix-valued semimartingales, now equipped with the pre-Lie product $\triangleright$. The algebra of polynomials over $\mathcal A$ is denoted $\mathbb{R}[\mathcal A]$ and we identify $m$-multilinear maps symmetric in the $m$ entries with maps from the degree $m$ component of this polynomial algebra. To avoid confusion between the product of two matrix-valued semimartingales in $\mathcal A$ and their (commutative) product in $\mathbb{R}[\mathcal A]$, we denote the latter $X\odot Y$. The brace map on $\mathcal A$ is the family of symmetric multilinear maps into $\mathcal A$ 
$$
 	\mathbb{R}[\mathcal A]\otimes \mathcal A\longrightarrow  \mathcal A, 
	\quad
	P\otimes X\longmapsto  \{P\}X,
$$
defined inductively by
$$
	\{Y\}X:=Y \triangleright X, 
$$
for $Y,X \in  \mathcal A$, and for $Y_1,Y_2,\ldots,Y_n,X \in \mathcal A$ we have 
$$
	\{Y_1,\dots,Y_n\}X:=\{Y_n\}(\{Y_1,\dots,Y_{n-1}\}X)-\sum\limits_{i=1}^{n-1}\{Y_1,\dots,\{Y_n\}Y_i,\dots,Y_{n-1}\}X.
$$
Observe that for $n=2$ the last equality encodes the pre-Lie identity \eqref{preLieid} as $\{Y_1,Y_2\}X=\{Y_2,Y_1\}X$. Following Guin and Oudom \cite{go2008}, we introduce a product $\ast$ on $\mathbb{R}[\mathcal A]$ in terms of the brace map. For elements $X_1,\dots,X_n$ and $Y_1,\dots,Y_m$ in $\mathcal A$,
\begin{equation}
\label{astproduct}
	(Y_1\odot \cdots \odot Y_m) \ast (X_1\odot \cdots \odot X_n)
	:=\sum\limits_{f} W_0\odot \{W_1\}X_1\odot \cdots \odot \{W_n\}X_n,
\end{equation}
where the sum is over all maps $f$ from $\{1,\ldots,m\}$ to $\{0,\ldots,n\}$ and the $W_i:=\prod_{j\in f^{-1}(i)} Y_j$. For example, $Y\ast X=YX+\{Y\}X$, for $X,Y \in \mathcal A$. 

Recall now that the enveloping algebra, $U (L)$, of a Lie algebra $L$ is an associative algebra (uniquely defined
up to isomorphism) such that \cite{Reutenauer1993}:

\begin{itemize}
\item the Lie algebra $L$ embeds in $U (L)$ (as a Lie algebra, where the Lie algebra structure on $U (L)$ is
induced by the associative product, that is, in terms of the ususal commutator bracket,
\item for any associative algebra $A$ (which is a Lie algebra, $L_A$, when equipped with
the commutator bracket), there is a natural bijection between Lie algebra maps from $L$
to $A$ and associative algebra maps from $U (L)$ to $A$.
\end{itemize}

The central result of the work of Oudom and Guin \cite{go2008} is the next theorem.

\begin{theorem}[\cite{go2008}]\label{thm:GO}
$\mathbb{R}[\mathcal A]$ with the product $\ast$ defined in \eqref{astproduct} is a non-commutative, associative and unital algebra. The product makes $\mathbb{R}[\mathcal A]$ the enveloping algebra of the Lie algebra $L_{\mathcal A}$ associated to $\mathcal A$.
\end{theorem}

Applying Theorem \ref{thm:GO} to $\mathcal A$ we see that the commutator bracket in $\mathcal A$ identifies with the pre-Lie bracket: $\llbracket X,Y\rrbracket =[\ ,\ ]_\triangleright$. On the other hand, by the universal properties of enveloping algebras, there is a unique associative algebra map $\iota$ from $(\mathbb{R}[A], \ast)$ to $\mathcal A$ which is the identity on $\mathcal A$.
In degree two we have: 
\begin{align*}
	\iota(Y\odot X)
	=\iota(Y\ast X)-\iota(\{Y\}X)
	&=YX - Y \triangleright X\\
	&=Y \preccurlyeq X+Y \succcurlyeq X - (Y \succcurlyeq X - X \preccurlyeq Y)\\
	&=Y \preccurlyeq X+X \preccurlyeq Y
	=: \mathcal{T}\langle Y,X \rangle,
\end{align*} 
where, using now the language of theoretical physics, we call time-ordered product of two elements in $\mathcal A$ the product $\mathcal{T}\langle Y,X \rangle:=X \preccurlyeq Y+Y \preccurlyeq X$. 
In general, for $X_1,\ldots,X_n\in \mathcal A $,
$$
	\mathcal{T}\langle X_1,X_2,\dots ,X_n\rangle:= \sum\limits_{\sigma\in S_n}
	X_{\sigma(1)} \preccurlyeq (X_{\sigma(2)} \preccurlyeq (\cdots  
	\preccurlyeq (X_{\sigma(n-1)} \preccurlyeq X_{\sigma(n)})\cdots )),
$$
where $S_n$ denotes the symmetric group of order $n$. The degree two calculation is a particular instance of a general phenomenon. The following Theorem relating pre-Lie products and time-ordered exponentials was obtained in \cite[p. 1291]{EFP2014}: 

\begin{theorem} \label{thm:keyth}
The image in $\mathcal A$ of a monomial $X_1\odot \cdots \odot X_n \in \mathbb{R}[\mathcal A]$ by the canonical map $\iota$ is the time-ordered product of the $X_i$s in $\mathcal A$:
 \begin{equation}
 \label{keyeq}
 	\iota (X_1\odot\cdots\odot X_n) = \mathcal{T}\langle X_1,\dots,X_n\rangle.
 \end{equation}
\end{theorem}

Notice that, in particular, 
$$
	\frac{1}{n!}\iota (X^{\odot n})
	=\frac{1}{n!}\mathcal{T}\langle X,\dots,X\rangle
	= X \preccurlyeq (X \preccurlyeq (\cdots  \preccurlyeq (X \preccurlyeq X)\cdots )).
$$
	
Let us apply these ideas to the study of the stochastic exponential and its logarithm in the Stratonovich framework. We address these problems at a purely formal level. Regarding the existence of the stochastic exponential and the convergence issues of the related series we refer to \cite{Protter2005} for the It\^o case and to Ben Arous \cite{Benarous1989} and Castell \cite{Castell1993} for the Stratonovich one.

\begin{definition}
For a continuous matrix-valued semimartingale, $X \in \mathcal A$, the (Stratonovich) right stochastic exponential is defined through $$\mathcal{E}_{\preccurlyeq}(X) = \mathbf{1} + \big(X \preccurlyeq \mathcal{E}_{\preccurlyeq}(X)\big),$$
or, by a Picard iteration, as a series
$$
	\mathcal{E}_{\preccurlyeq}(X)
	=\mathbf{1}+X+X \preccurlyeq X+\dots 
	+X \preccurlyeq (X \preccurlyeq (\cdots  \preccurlyeq (X \preccurlyeq X)\cdots ))+\cdots
$$
\end{definition}

We are interested in the stochastic analogue of the classical Baker--Campbell--Hausdorff problem of computing the logarithm $\Omega(X)$ of the solution
\begin{equation}
\label{eq:preLieMag2}
 	\Omega(X)=\log\!\big(\mathcal{E}_{\preccurlyeq}(X)\big).
\end{equation}
Since Stratonovich calculus obeys the usual integration by parts rule, the well-known Strichartz formula holds \cite{Benarous1989,MielPleb1970,Strichartz1987}. We are interested here in the stochastic analog of the so-called Magnus solution.

By Theorem \ref{thm:keyth}, the equation 
$$\mathcal{E}_{\preccurlyeq}(X)=\exp\!\big(\Omega(X)\big)$$
lifts in $\mathbb{R}[\mathcal A]$ to an equality of exponentials:
\begin{equation}
\label{eq:preLieMag4}
 	\exp^{\odot} (X)=\exp^{\ast}( \tilde\Omega(X)),
\end{equation}
where $\exp^{\odot}(X)$ (resp.~$\exp^{\ast}(\tilde\Omega(X))$) denotes the exponential of $X \in \mathcal A$ (resp.~$\tilde\Omega(X)$) for the $\odot$ (resp.~$\ast$) product. Theorem \ref{thm:keyth} together with the general properties of enveloping algebras insure that this identity maps to (\ref{eq:preLieMag2}) by $\iota$ and that $\iota(\tilde\Omega(X))=\Omega(X)$. The next proposition was shown in \cite{ChaPa2013}, using the fact that the maps $\exp^\odot$ and $\exp^\ast$ have a Lie theoretic interpretation.
 
\begin{proposition}\label{ChaPa2013}
The element $\tilde\Omega(X)=\log^\ast\circ \exp^\odot(X)$ belongs to $\mathcal A$ and satifies the fixed point equation:
\begin{equation}
\label{MagLietheo}
	\tilde\Omega(X)=\big\lbrace\frac{\tilde\Omega(X)}{\exp^\ast(\tilde\Omega(X))-{\mathbf 1}}\big\rbrace X,
\end{equation}
where $\tilde\Omega(X)/(\exp^\ast(\tilde\Omega(X))-{\mathbf 1})$ is computed in $\mathbb{R}[\mathcal A]$ using the $\ast$ product.
\end{proposition}

Note that the brace map is in place on the righthand side in \eqref{MagLietheo}. We set
$$
	\ell^{(n)}_{X\triangleright}(Y) :=X \triangleright (\ell^{(n-1)}_{X\triangleright}(Y)),\ \ \ell^{(0)}_{X\triangleright}(Y):=Y.
$$ 
The $B_n/n!$ are the coefficients of the formal power series expansion of $x/(\exp(x)-1)$ and, by formal properties of the enveloping algebra, we have $\iota(\lbrace X^{\ast n}\rbrace Y)=\ell^{(n)}_{X\triangleright}(Y)$. We refer, e.g., to \cite{Nacer2020} for an explanation of this general phenomenon in the context of enveloping algebras of pre-Lie algebras. We recover finally the pre-Lie Magnus expansion of the logarithm of the right Stratonovich exponential obtained in \cite{CKP2018}.

\begin{theorem} The continuous matrix-valued semimartingale
$\Omega(X)$ satisfies the fixed point equation
\begin{equation}
\label{eq:preLieMag3}
	\Omega(X) = \sum_{n \ge 0} \frac{B_n}{n!} \ell^{(n)}_{\Omega(X)\triangleright}(X).
\end{equation}
\end{theorem}

\begin{remark}
The left stochastic exponential is defined similarly through $\mathcal{E}_{\succcurlyeq}(X)_t = \mathbf{1} + \big(\mathcal{E}_{\succcurlyeq}(X) \succcurlyeq X\big)_t.$
It satisfies 
\begin{equation}
\label{eq:preLieMag1}
 	\mathcal{E}_{\succcurlyeq}(X)=\exp\!\big(\!-\Omega(-X)\big).
\end{equation}
\end{remark}


\subsection{Chronological It\^o calculus}
\label{ssec:prelieIto}

In the present subsection we will apply the machinery developed in the previous subsection in the context of It\^o calculus. As remarked earlier, our arguments are purely algebraic. We do neither address the question of existence of stochastic exponentials nor do we discuss convergence issues of the associated series. On these questions, the reader is referred to the standard reference \cite{Protter2005}.

The existence of a continuous Baker--Campbell--Hausdorff, or Strichartz formula, the presence of pre-Lie structures as well as a Magnus formula could be expected in Stratonovich calculus due to the fact that the latter satisfies the usual rules of calculus. In It\^o calculus, however, things are not so simple due to the presence of the covariation bracket and the fact that the usual shuffle algebra structure must be replaced by Karandikar's axioms, i.e., a quasi-shuffle algebra. For iterated It\^o integrals of matrix-valued semimartingales, a Strichartz-type formula was obtained in \cite{EMPW2015a,EMPW2015b}. The difference with the classical formula reflect the fact that one must take into account the covariation bracket. This is achieved by replacing bijections, that is, permutations and their descent statistics as they appear in the classical formula, by surjections and a suitable notion of descents in this new context. 

Here, we focus again on pre-Lie structures and the Magnus formula in the context of It\^o calculus. Our results are obtained by adapting ideas from the theory of Rota--Baxter algebras to quasi-shuffle algebras. Our presentation is almost self-contained. On Rota--Baxter algebras and integral calculus, we refer to \cite{Nacer2020} for a general survey combined with  references.

In this subsection, $\mathcal A$ denotes an algebra of matrix-valued semimartingales (notice that we do not require continuity anymore). 

\begin{proposition} For $X,Y \in \mathcal A$, set $X\succdot Y:=X\succ Y+[X,Y]$ and $X \blacktriangleright\ Y:=X\succdot  Y-Y\prec X$, then the pair $(\mathcal A,\blacktriangleright)$ is a pre-Lie algebra. Moreover, $\llbracket X,Y \rrbracket =[X,Y]_\blacktriangleright$.
\end{proposition}

\begin{proof}
Indeed, the quasi-shuffle axioms imply that $$XY=X\prec Y+X\succ Y+[X,Y]=X\prec Y+X \succdot Y.$$ This yields
\begin{align*}
	\llbracket X,Y \rrbracket 
	= XY-YX 
	&= X\prec Y+X\succdot Y-Y\prec X-Y\succdot X\\
	&=(X\succdot  Y-Y\prec X)-(Y\succdot  X-X\prec Y)\\
	&=X\blacktriangleright Y-Y\blacktriangleright X=[X,Y]_\blacktriangleright.
\end{align*}
Using the Jacobi identity (to avoid any notational ambiguity, recall that $[\ ,\ ]$ denotes in this article the covariation bracket, not to be confused with the Lie bracket $\llbracket \ ,\  \rrbracket$):
\begin{align*}
	([X,Y]_\blacktriangleright\ \blacktriangleright Z)_t
	&=(\llbracket X,Y \rrbracket \blacktriangleright Z)_t
	=\int_0^t \llbracket \llbracket  X_s,Y_s\rrbracket ,\circ dZ_s\rrbracket
				+[\llbracket X,Y\rrbracket,Z]_t\\
	&=(\int_0^t\llbracket X_s,\llbracket Y_s,\circ dZ_s\rrbracket \rrbracket  +[ XY ,Z]_t)
	 -(\int_0^t\llbracket Y_s,\llbracket  X_s,\circ dZ_s\rrbracket \rrbracket  +[ XY ,Z]_t)\\
	&=\int_0^t\llbracket  X_s,\llbracket  Y_s,\circ dZ_s\rrbracket \rrbracket 
			+[X\succ Y+X\prec Y+[X,Y],Z]_t\\
    	& -(\int_0^t\llbracket  Y_s,\llbracket X_s,\circ dZ_s\rrbracket \rrbracket 
			+[ Y\succ X+Y\prec X+[Y,X] ,Z]_t)\\
    	&=\int_0^t\llbracket X_s,\llbracket Y_s,\circ dZ_s\rrbracket \rrbracket 
			+[X, Y\succ Z]_t+(X\succ [Y,Z])_t+[[X,Y],Z]_t\\
    	&-(\int_0^t\llbracket  Y_s,\llbracket  X_s,\circ dZ_s\rrbracket \rrbracket  
    			+ [Y, X\succ Z]_t+(Y\succ [X,Z])_t+[[Y,X],Z]_t)\\
	&=(X\blacktriangleright (Y\blacktriangleright Z))_t 
			- (Y\blacktriangleright (X\blacktriangleright Z))_t.
\end{align*}
\end{proof}

\begin{remark}The triple $(\mathcal A,\prec,\succdot)$ is a shuffle algebra. More generally, in \cite{K2002} it was shown that any quasi-shuffle algebra gives rise to a shuffle algebra.  
Indeed, $(X\prec Y)\prec Z=X\prec (YZ)$ and
\begin{align*}
	(X\succdot Y)\prec Z 
	&=(X\succ Y)\prec Z+[X,Y]\prec Z\\
	&=X\succ(Y\prec Z)+[X,Y\prec Z]
	=X\succdot (Y\prec Z)\\[0.3cm]
	(XY)\succdot Z
	&=(X Y)\succ Z+[XY,Z]\\
	&=(X Y)\succ Z +[X\succ Y,Z] + [X\prec Y,Z]+[[X,Y],Z]\\
	&=X\succ(Y\succ Z)+ X\succ [Y,Z]+[X,Y\succ Z]+[X,[Y,Z]]\\
	&=X\succ (Y\succdot Z)+[X,Y\succdot Z]
	=X\succdot (Y\succdot Z).
\end{align*}
We note that this property is also common in Rota--Baxter algebras, where a shuffle algebra structure is defined similarly starting from the operations $R(X)Y$, $XR(Y)$ and $XY$ instead of $\prec,\succ,[-,- ]$. 
\end{remark}

\begin{definition}
For a matrix-valued semimartingale $X \in \mathcal A$ the (It\^o) right stochastic exponential is defined through $$\mathcal{E}_{\prec}(X) = \mathbf{1} + \big(X \prec \mathcal{E}_{\prec}(X)\big),$$
or, by a Picard iteration, as a series
$$
	\mathcal{E}_{\prec}(X)
	=\mathbf{1}+X+X \prec X+\dots +X \prec (X \prec (\cdots  \prec (X \prec X)\cdots ))+\dots
$$
\end{definition}

We are interested again in the stochastic analogue of the Baker--Campbell--Hausdorff problem of computing the logarithm $\Gamma(X)$ of the solution:
\begin{equation}
\label{eq:preLieMag2bis}
 	\mathcal{E}_{\prec}(X)=\exp\!\big(\Gamma(X)\big).
\end{equation}

Let us denote now by $\mathbb{R}_\blacktriangleright[\mathcal A]$ the enveloping algebra of $\mathcal A$ constructed exactly as in the previous section but using the new pre-Lie product $\blacktriangleright$ instead of $\triangleright$. In particular, as a vector space, $\mathbb{R}_\blacktriangleright[\mathcal A]=\mathbb{R}[\mathcal A]$, the algebra of polynomials over $\mathcal A$. 

To avoid notational ambiguities, we write $\ast_\blacktriangleright$ for the associative product making $\mathbb{R}_\blacktriangleright[\mathcal A]$ the enveloping algebra of $L_\mathcal A$ and $\{P\}_\blacktriangleright X$ the brace map for $P \in \mathbb{R}_\blacktriangleright[\mathcal A]$. We also write $\mathcal{T}_\blacktriangleright$ for the corresponding time-ordered product, associated to the left-half shuffle $\prec$, e.g., $\mathcal{T}_\blacktriangleright[X,Y]:=X\prec Y$, and so on. Lastly, we denote $\iota_\blacktriangleright$ the canonical algebra map from $(\mathbb{R}_\blacktriangleright[\mathcal A],\ast_\blacktriangleright)$ to $\mathcal A$ obtained from the universal properties of the enveloping algebra. Theorem \ref{thm:keyth} holds \it mutatis mutandis \rm in the new context and equation (\ref{eq:preLieMag2bis}) lifts in $\mathbb{R}_\blacktriangleright[\mathcal A]$ to:
\begin{equation}
\label{eq:preLieMag4bis}
 	\exp^{\odot} (X)=\exp^{\ast_\blacktriangleright}( \tilde\Gamma(X)).
\end{equation}
Theorem \ref{thm:keyth} and the general properties of enveloping algebras insure that this identity maps to (\ref{eq:preLieMag2bis}) by $\iota_\blacktriangleright$ and that $\iota_\blacktriangleright(\tilde\Gamma(X))=\Gamma(X)$.

We obtain finally the analogous of Proposition \ref{ChaPa2013}:

\begin{proposition}The element $\tilde\Gamma(X)=\log^\ast_\blacktriangleright\circ \exp^\odot(X)$ belongs to $\mathcal A$ and satifies the fixed point equation:
\begin{equation}\label{MagLietheobis}
	\tilde\Gamma(X)
	=\big\lbrace\frac{\tilde\Gamma(X)}{\exp^\ast_\blacktriangleright(\tilde\Gamma(X))
	-{\mathbf 1}}\big\rbrace_\blacktriangleright X,
\end{equation}
where $\tilde\Gamma(X)/(\exp^\ast_\blacktriangleright(\tilde\Gamma(X))-{\mathbf 1})$ is computed in $\mathbb{R}_\blacktriangleright[\mathcal A]$ using the $\ast_\blacktriangleright$ product.
\end{proposition}

Setting 
$$
	\ell^{(n)}_{X\blacktriangleright}(Y) :=X \blacktriangleright (\ell^{(n-1)}_{X\blacktriangleright}(Y)),
$$
$\ell^{(0)}_{X\blacktriangleright}(Y):=Y$, we get finally a pre-Lie Magnus expansion of the logarithm of the right It\^o stochastic exponential:

\begin{theorem} 
The matrix-valued semi-martingale  $\Gamma(X)$, which is the logarithm of the (It\^o) stochastic exponential, satisfies the fixed point equation
\begin{equation}
\label{eq:preLieMag3bis}
	\Gamma(X) = \sum_{n \ge 0} \frac{B_n}{n!} \ell^{(n)}_{\Gamma(X)\blacktriangleright}(X).
\end{equation}
\end{theorem}


%
%
%

\end{document}